\documentclass[12pt]{amsart}
\usepackage{amsmath,amsfonts,amsthm,amssymb}
\usepackage{graphicx}
\usepackage{mathrsfs}

\newcommand{\R}{\mathbb{R}}

\newcommand{\N}{\mathbb{N}}

\newcommand{\e}{\varepsilon}

\newcommand{\qqand}{\qquad\text{and}\qquad}

\newtheorem{theorem}{Theorem}[section]
\newtheorem{corollary}[theorem]{Corollary}
\newtheorem{lemma}[theorem]{Lemma}
\newtheorem{proposition}[theorem]{Proposition}

\theoremstyle{definition}

\theoremstyle{remark}
\newtheorem{remark}{Remark}[section]

\DeclareMathOperator{\supp}{supp}

\DeclareMathOperator{\pv}{P.V.}

\numberwithin{equation}{section}

\makeindex

\begin{document}

\title[Strongly singular operators]{Estimates for 
strongly singular operators along curves}

\author{Magali Folch-Gabayet}
\address{Instituto de Matem\'aticas\\
Universidad Nacional Aut\'onoma de M\'exico\\
Circuito Exterior, Cd. Universitaria\\Mexico City, Mexico 04510}
\email{folchgab@matem.unam.mx}

\author{Ricardo A. S\'aenz}
\address{Facultad de Ciencias\\
Universidad de Colima\\
Ave. Bernal D\'iaz del Castillo \# 340, Col. Villa San Sebasti\'an\\
Colima, Colima, Mexico 28045}
\email{rasaenz@ucol.mx}


\begin{abstract}
For a proper function $f$ on the plane, we study the operator
\[
Tf(x,y) = \lim_{\e\to 0} \int_\e^1 f(x-t,y-t^k) \frac{e^{2\pi i \gamma(t)}}{\psi(t)} dt,
\]
where $k\ge1$ and $\psi$ and $\gamma$ are functions defined near the origin such
that $\psi(t)\to 0$ and $|\gamma(t)|\to\infty$ as $t\to 0$. We give sufficient
regularity and growth conditions on $\psi$ and $\gamma$ for its multiplier to be
a bounded function, and thus for the operator to be bounded on $L^2(\R^2)$. We consider an extension to $L^p(\R^2)$, for certain $p's$.
\end{abstract}

\maketitle

\section{Preliminaires}

In trying to generalize the Calder\'on-Zygmund method of rotations to kernels with
non standard homogeneity, you are led to the study of the Hilbert transform
along curves. This work was started by Fabes and Rivi{\`e}re in 
\cite{FabesRiviere}. See \cite{SteinBeijing} for a very nice explanation of the
difficulties on dealing with such operators.

It is known that
\begin{equation*}
H_\phi f(x) = \pv \int_{-1}^1 f(x - \phi(t)) \frac{dt}{t}
\end{equation*}
is bounded on $L^p$ for certain well-curved $\phi$ in $\R^d$ (see
\cite{NRW1,NRW2,SteinWainger,CNS99}). However, when the singularity is given by
a power greater than one, say $t|t|^\alpha$ with $\alpha>0$, Chandarana
\cite{Chandarana} proved that the operator is bounded as long as a sufficiently
rapid oscillatory term is involved. Indeed, Chandarana proved that the operator
\begin{equation}\label{Tab}
T_{\alpha,\beta}f(x,y) = \pv \int_{-1}^1 
f(x - \phi(t))  \frac{e^{2\pi i|t|^{-\beta}}}{t|t|^\alpha}dt,
\end{equation}
for the curve $\phi(t) = (t,t^k)$, $k\ge 3$, is bounded on $L^2(\R^2)$ if and
only if
\begin{equation*}
\beta \ge 3\alpha.
\end{equation*}
The author also proved that the operator $T_{\alpha,\beta}$ is bounded on 
$L^p(\R^2)$ as long as $\beta>3\alpha$ and 
\begin{equation*}
\Big| \frac{1}{p} - \frac{1}{2} \Big| < 
\frac{(\beta-3\alpha)(\beta+2)}{(\beta-3\alpha)(\beta+2) + 12\alpha(\beta+1)}.
\end{equation*}

In this work we discuss the corresponding hyper-singular operator in $\R^2$
along the polynomial curve $\phi(t) =  (t, t^k)$ with a singularity much worse
than a power, say of the order of $t\to e^{-1/|t|}$ or worse. The authors have
studied oscillatory integrals with such kernels previously in 
\cite{Folch99} and \cite{FolchSaenz19}.

We now consider, for $k\ge 1$, the operator
\begin{equation}\label{operator}
Tf(x,y) = \lim_{\e\to 0} \int_\e^1 f(x-t,y-t^k) \frac{e^{2\pi i \gamma(t)}}{\psi(t)} dt,
\end{equation}
where the functions $\gamma\in C^3((0,1])$ and $\psi\in C^2((0,1])$ satisfy
the following assumptions:
\begin{enumerate}
\item[(a.1)] The functions $\gamma, \psi$, and their derivatives 
$\gamma', \gamma'', \gamma'''$, and $\psi'$ are all monotone. We also
assume $\gamma''(t) > 0$ for $t\in(0,1]$, that $\gamma''$ is decreasing, 
$\lim_{t\to0}\gamma''(t) = \infty$ and that $\gamma''(1) = 1$.
\item[(a.2)] There exists $C>0$ such that, for $s\ge 1$, 
\begin{equation*}
|\gamma'''(\gamma''^{-1}(2s))| \le C |\gamma'''(\gamma''^{-1}(s))|.
\end{equation*}
\item[(a.3)] For small $\rho>0$, there exists $C>0$ such that
\begin{equation*}
|\gamma'''(t)| \le C \gamma''(t)^{3/2-\rho}.
\end{equation*}
\item[(a.4)] There exists $C>0$ such that
\begin{equation*}
\frac{1}{|\psi(t)|} \le C |\gamma'''(t)|^{1/3}.
\end{equation*}
\end{enumerate}

These assumptions are satisfied by $\gamma(t) = t^{-\beta}$ and 
$\psi(t) = t^{\alpha+1}$ if $\alpha$ and $\beta$ satisfy, as in 
\cite{Chandarana}, $\beta \ge 3\alpha$. Assumption (a.4) is not satisfied by our
example $\gamma(t) = e^{1/t}$ and $\psi(t) = e^{-1/2t}$ in 
\cite{FolchSaenz19}, but it is satisfied by the function $\psi(t) = e^{-1/3 t}$.
Alternatively, it is satisfied by the pair of functions 
$\gamma(t) = e^{3/t}$ and 
$\psi(t) = e^{-1/t}$.


\subsection{Existence}

It is not hard to see that the limit in \eqref{operator} exists for, say, 
$f\in C_c^\infty(\R^2)$. Indeed, if we write 
\begin{equation*}
\Gamma(s) = \int_0^s e^{2\pi i \gamma(t)} dt,
\end{equation*}
for $0 < \e' < \e$ we can integrate by parts to obtain
\begin{multline}\label{exis-parts}
\int_{\e'}^\e f(x-t,y-t^k) \frac{\Gamma'(t)}{\psi(t)} dt = \\
f(x-t,y-t^k) \frac{\Gamma(t)}{\psi(t)}\bigg|_{\e'}^\e -
\int_{\e'}^\e \Gamma(t) \frac{d}{dt}  \frac{f(x-t,y-t^k)}{\psi(t)} dt.
\end{multline}
The first term in \eqref{exis-parts} goes to zero as $\e',\e\to 0$ because, by
the van der Corput's lemma \cite[Proposition 2 in VIII.1.2]{Stein},
\begin{equation*}
|\Gamma(t)| \lesssim \frac{1}{\gamma''(t)^{1/2}},
\end{equation*}
and therefore, using assumptions (a.3) and (a.4) and the fact that $f$ is 
bounded,
\begin{equation*}
\bigg|f(x-t,y-t^k) \frac{\Gamma(t)}{\psi(t)}\bigg| \lesssim
\frac{1}{\gamma''(t)^{1/2}}|\gamma'''(t)|^{1/3} \lesssim
\frac{1}{\gamma''(t)^\rho} \to 0
\end{equation*}
as $t\to 0$.

For the second integral in \eqref{exis-parts}, we note that
\begin{multline}\label{exis-integral}
\bigg|\int_{\e'}^\e \Gamma(t) \frac{d}{dt}  \frac{f(x-t,y-t^k)}{\psi(t)} dt
\bigg| \le \\
\int_{\e'}^\e |\Gamma(t)| \Big|\frac{df(x-t,y-t^k)}{dt}\Big| 
\frac{1}{|\psi(t)|} dt \\ +
\int_{\e'}^\e |\Gamma(t)| |f(x-t,y-t^k)|
\Big|\frac{\psi'(t)}{\psi(t)^2}\Big| dt,
\end{multline}
and the first integral in \eqref{exis-integral} goes to 0 as above, because
the derivatives of $f$ are also bounded. For the second integral, choose 
$l\in\N$ such that $\gamma''^{-1}(2^{l+1}) < \e \le \gamma''^{-1}(2^l)$. As
$0 < \e' < \e \le \gamma''^{-1}(2^l)$, we have, using the fact that $f$ is 
bounded,
\begin{equation*}
\begin{split}
\int_{\e'}^\e |\Gamma(t)| &|f(x-t,y-t^k)|
\Big|\frac{\psi'(t)}{\psi(t)^2}\Big| dt\\
&\lesssim 
\int_0^{\gamma''^{-1}(2^l)} |\Gamma(t)|\Big|\frac{\psi'(t)}{\psi(t)^2}\Big| dt
\lesssim 
\sum_{j=l}^\infty \frac{1}{2^{j/2}} 
\int_{\gamma''^{-1}(2^{j+1})}^{\gamma''^{-1}(2^j)} 
\Big|\frac{\psi'(t)}{\psi(t)^2}\Big| dt,
\end{split}
\end{equation*}
where we have again used the van der Corput's lemma so
\begin{equation*}
|\Gamma(\gamma''^{-1}(2^l))| \lesssim
\frac{1}{\gamma''(\gamma''^{-1}(2^l))^{1/2}} = \frac{1}{2^{j/2}}.
\end{equation*}
Now, as $\psi$ is monotone, we have
\begin{equation*}
\begin{split}
\int_{\gamma''^{-1}(2^{j+1})}^{\gamma''^{-1}(2^j)} 
\Big|\frac{\psi'(t)}{\psi(t)^2}\Big| dt &= 
\bigg| \frac{1}{\psi(\gamma''^{-1}(2^{j+1}))} - 
\frac{1}{\psi(\gamma''^{-1}(2^j))} \bigg| \\
&\lesssim |\gamma'''(\gamma''^{-1}(2^{j+1}))|^{1/3},
\end{split}
\end{equation*}
where we have also used (a.4). Using (a.3), we see that the integral is then
bounded by
\begin{equation*}
\sum_{j=l}^\infty \frac{1}{2^{j/2}} \big( 2^{j/3}\big)^{3/2-\rho}
= \sum_{j=l}^\infty \frac{1}{2^{\rho/3}} \lesssim \frac{1}{2^l}, 
\end{equation*}
which goes to 0 as $l\to\infty$ (and thus as $\e',\e\to 0$).


\section{$L^2$ boundedness}

We will prove the following theorem.

\begin{theorem}\label{thm}
The operator $T$ given by \eqref{operator} extends to a bounded operator on
$L^2(\R^2)$.
\end{theorem}

In order to prove the $L^2$ boundedness of operator \eqref{operator} we need to
prove that its multiplier
\begin{equation}\label{multiplier}
m(\xi,\eta) = \int_0^1 e^{2\pi i (\gamma(t) - \xi t - \eta t^k)}
\frac{dt}{\psi(t)}, \quad \xi,\eta\in\R,
\end{equation}
is bounded. We write \eqref{multiplier} as the integral
\begin{equation}\label{mult-int}
\int_0^1 e^{2\pi i g(t)}\frac{dt}{\psi(t)},
\end{equation}
where $g(t) = \gamma(t) - \xi t - \eta t^k$. If we let 
\begin{equation*}
G(s) = \int_0^s e^{2\pi i g(t)} dt,
\end{equation*}
then we can also write \eqref{mult-int} as
\begin{equation*}
\int_0^1 \frac{G'(t)}{\psi(t)} dt.
\end{equation*}

The boundedness of \eqref{mult-int} will follow from estimates to $G(s)$,
obtained by van de Corput's lemma 
provided we can prove that the derivatives of $g(t)$ can be bounded from below,
uniformly for $\xi$ and $\eta$ in proper subsets of $\R^2$. The derivatives of
the function $g$ are given by
\begin{equation*}
g'(t) = \gamma'(t) - \xi - \eta kt^{k-1},
\end{equation*}
\begin{equation*}
g''(t) = \gamma''(t) - \eta k(k-1)t^{k-2},
\end{equation*}
and
\begin{equation*}
g'''(t) = \gamma'''(t) - \eta k(k-1)(k-2) t^{k-3}.
\end{equation*}
We thus observe that bounds from below for $g'$, $g''$ or $g'''$ will depend
on the signs of $\xi$ and $\eta$. We prove then the boundedness of
\eqref{mult-int} by splitting in such defined cases.


\subsection{Case $\eta \le 0$}
\label{section-neg-eta}

In this case we consider the second derivative $g''$. For $\eta \le 0$,
\begin{equation*}
g''(t) = \gamma''(t) - \eta k(k-1)t^{k-2} \ge \gamma''(t)
\end{equation*}
Hence, by van der Corput's lemma,
\begin{equation*}
|G(t)| \lesssim \frac{1}{\gamma''(t)^{1/2}}.
\end{equation*}
We integrate by parts to obtain
\begin{equation*}
\int_0^{1}\frac{G'(t)}{\psi(t)}dt =
\lim_{t\to0} \Big( \frac{G(1)}{\psi(1)}
- \frac{G(t)}{\psi(t)}\Big) + 
\int_0^{1} \frac{G(t)\psi'(t)}{\psi(t)^2}dt.
\end{equation*}
The limit is bounded because, using assumptions (a.3) and (a.4),
\begin{equation*}
\Big|\frac{G(t)}{\psi(t)}\Big| \lesssim 
\frac{1}{\gamma''(t)^{1/2}} |\gamma'''(t)|^{1/3} \lesssim
\frac{1}{\gamma''(t)^{\rho/3}} \le 1.
\end{equation*}
For the integral, we partition the interval 
$(0, 1]$ in the intervals
\begin{equation*}
\gamma''^{-1}\big((2^j, 2^{j+1}]\big), 
\end{equation*}
$j=0,1,2\ldots$, and we use the monotonicity of $\psi$, so $\psi'$
doesn't change and thus
\begin{equation*}
\begin{split}
\Big| \int_0^{1} \frac{G(t)\psi'(t)}{\psi(t)^2}dt \Big|
&\le 
\sum_{j=0}^\infty 
\int_{\gamma''^{-1}(2^{j+1})}^{\gamma''^{-1}(2^j)} |G(t)|
\Big| \frac{\psi'(t)}{\psi(t)^2} \Big| dt \\
&\lesssim \sum_{j=0}^\infty 
\int_{\gamma''^{-1}(2^{j+1})}^{\gamma''^{-1}(2^j)} \frac{1}{\gamma''(t)^{1/2}}
\Big| \frac{\psi'(t)}{\psi(t)^2} \Big| dt \\
&\le \sum_{j=0}^\infty \frac{1}{2^{j/2}}
\int_{\gamma''^{-1}(2^{j+1})}^{\gamma''^{-1}(2^j)} 
\Big| \frac{\psi'(t)}{\psi(t)^2} \Big| dt \\
&\le \sum_{j=0}^\infty \frac{1}{2^{j/2}}
\Big| \frac{1}{\psi(\gamma''^{-1}(2^j))} -
\frac{1}{\psi(\gamma''^{-1}(2^{j+1}))} \Big| \\
&\lesssim \sum_{j=0}^\infty \frac{1}{2^{j/2}}
|\gamma'''(\gamma''^{-1}(2^{j+1}))|^{1/3}\\
&\lesssim \sum_{j=0}^\infty \frac{1}{2^{j/2}} \big( 2^{j(3/2-\rho)} \big)^{1/3}
= \sum_{j=0}^\infty \frac{1}{2^{j\rho/3}} < \infty.
\end{split}
\end{equation*}

Note that, implicitly, we have assumed $k\ge 2$. The case $k=1$ provides no
difficulty, because in that case we have $g''(t) = \gamma''(t)$.

\subsection{Case $\eta > 0$.}

If $k \ge 2$ and $\eta > 0$, both $g'$ and $g''$ might be zero. Indeed, since
\begin{equation*}
g''(t) = \gamma''(t) - \eta k(k-1)t^{k-2},
\end{equation*}
$\gamma''>0$ and decreasing, and $\eta > 0$, then $-\eta t^{k-2}$ is also
decreasing and $g''$ will have a unique zero in $(0,1]$ if 
$\eta k(k-1) > \gamma''(1) = 1$.

Let $t_0$ be the zero of $g''(t)$ if $\eta k(k-1) > 1$; otherwise
let $t_0 = 1$. Hence $t_0$ satisfies
\begin{equation*}
\gamma''(t_0) = \eta k(k-1)t_0^{k-2}
\end{equation*}
if $t_0 < 1$, or $1 \ge \eta k(k-1)$ if $t_0 = 1$. Either way we have the
estimate 
\begin{equation}\label{t0est}
\gamma''(t_0) \ge \eta k(k-1)t_0^{k-2}.
\end{equation}

Let $l\in\N$ be such that
\begin{equation}\label{t0l}
\gamma''^{-1}(2^{l+1}) < t_0 \le \gamma''^{-1}(2^l).
\end{equation}
We thus split integral \eqref{mult-int} as
\begin{equation*}
\int_0^{\gamma''^{-1}(2^{l+2})} \frac{G'(t)}{\psi(t)}dt
+ \int_{\gamma''^{-1}(2^{l+2})}^{\gamma''^{-1}(2^{l-1})} \frac{G'(t)}{\psi(t)}dt
+ \int_{\gamma''^{-1}(2^{l-1})}^1 \frac{G'(t)}{\psi(t)}dt,
\end{equation*}
where the third integral is nonzero only if $l>1$ (if $l=0$, the second integral
is only over the interval $[\gamma''^{-1}(4),1]$, of course). We estimate each
of these integrals separately.

\subsubsection{The first integral}
\label{first-int}

Here we integrate over $t$ such that $0 < t \le \gamma''^{-1}(2^{l+2})$. For
such $t$ we have, using the monotonicity of $\gamma''$, \eqref{t0est} and \eqref{t0l},
\begin{equation*}
\begin{split}
\gamma''(t) &\ge 2^{l+2} = 2\cdot 2^{l+1} = 2\gamma''(\gamma''^{-1}(2^{l+1}))\\
&\ge 2\gamma''(t_0) \ge 2\eta k(k-1)t_0^{k-2} \ge 2\eta k(k-1)t^{k-2}.
\end{split}
\end{equation*}
Hence
\begin{equation*}
g''(t) \ge \frac{1}{2}\gamma''(t)
\end{equation*}
over $0 < t \le \gamma''^{-1}(2^{l+2})$, and by van der Corput's lemma,
\begin{equation*}
|G(t)| \lesssim \frac{1}{\gamma''(t)^{1/2}}
\end{equation*}
over this interval. We can therefore estimate the first integral as in 
\ref{section-neg-eta}.

\subsubsection{The third integral}
\label{third-int}

If $l > 1$, we integrate over $\gamma''^{-1}(2^{l-1}) \le t \le 1$. This implies
that $t_0 < 1$ so, in particular, $\gamma''(t_0) = \eta k(k-1)t_0^{k-2}$. Using
the monotonicity of $\gamma''$ and \eqref{t0l},
\begin{equation*}
\begin{split}
\eta k(k-1)t^{k-2} &\ge  \eta k(k-1) t_0^{k-2} = \gamma''(t_0) \\
&\ge 2^l = 2 \gamma''(\gamma''^{-1}(2^{l-1})) \ge \gamma''(t),
\end{split}
\end{equation*}
so we have again, by van der Corput's lemma,
\begin{equation*}
|G(t)| \lesssim \frac{1}{\gamma''(t)^{1/2}}
\end{equation*}
and we proceed as above.

\subsubsection{The middle integral}
\label{middle-int}

We finally integrate over the interval that contains $t_0$. Since $-\eta$ is
negative, we have
\begin{equation*}
g'''(t) = \gamma'''(t) - \eta k(k-1)(k-2)t^{k-3} < \gamma'''(t),
\end{equation*}
or $g'''(t) = \gamma'''(t)$ if $k=2$.
Thus
\begin{equation*}
|g'''(t)| \ge |\gamma'''(t)|.
\end{equation*}
for $\gamma''^{-1}(2^{l+2}) \le t \le \gamma''^{-1}(2^{l-1})$. By the van der
Corput's lemma and hypothesis (a.4) we obtain
\begin{equation*}
\begin{split}
\Big| 
\int_{\gamma''^{-1}(2^{l+2})}^{\gamma''^{-1}(2^{l-1})}
\frac{e^{2\pi i g(t)}}{\psi(t)}dt \Big| &\lesssim
\frac{1}{|\gamma'''(\gamma''^{-1}(2^{l-1}))|^{1/3}}
\frac{1}{|\psi(\gamma''^{-1}(2^{l+2}))|}\\
&\lesssim
\frac{1}{|\gamma'''(\gamma''^{-1}(2^{l-1}))|^{1/3}}
|\gamma'''(\gamma''^{-1}(2^{l+2}))|^{1/3}
\end{split}
\end{equation*}
which is bounded by a constant by assumption (a.2).

\section{A result on sharpness}

Assume that $\gamma$ satifies assumptions (a.1)-(a.3). However, instead
of assumption (a.4), we now have
\begin{enumerate}
\item[(b.4)]
There exists $C>0$ such that, for $t\in(0,1]$,
\begin{equation*}
\frac{1}{|\psi(t)|} \le C \gamma''(t)^{1/2 - \rho},
\end{equation*}
where $\rho$ can be taken to be the same as in (a.3) above.
\item[(b.5)] 
The quotient
\begin{equation*}
\frac{1}{|\psi(t)| |\gamma'''(t)|^{1/3}}
\end{equation*}
is unbounded.
\end{enumerate}

\begin{proposition}
Under the previous hypotheses, with $k=2$, the multiplier $m(\xi, \eta)$ for
the parabola is unbounded.
\end{proposition}

\begin{proof}
Let $t_n\to0$ such that 
\begin{equation*}
\frac{1}{|\psi(t_n)| |\gamma'''(t_n)|^{1/3}} \to \infty.
\end{equation*}
Now, for each $n$, set 
\begin{equation*}
\eta_n = \frac{1}{2}\gamma''(t_n) \qqand
\xi_n = \gamma'(t_n) - \gamma''(t_n) t_n.
\end{equation*}
Note that this implies that $g'(t_n) = g''(t_n) = 0$. We claim that
$|m(\xi_n,\eta_n)| \to \infty$.

Fix $n\in\N$ and choose $l = l(n)\in\N$ such that
\begin{equation*}
\gamma''^{-1}(2^{l+1}) < t_n \le \gamma''^{-1}(2^l).
\end{equation*}
Let $h\in C^\infty((0,1))$ be a cutoff function such that
\begin{enumerate}
\item $\supp h \subset [\gamma''^{-1}(2^{l+3}),\gamma''^{-1}(2^{l-2})]$;
\item $h(t) = 1$ on $[\gamma''^{-1}(2^{l+2}),\gamma''^{-1}(2^{l-1})]$; and
\item $h$ is increasing on $[\gamma''^{-1}(2^{l+3}),\gamma''^{-1}(2^{l+2})]$
and decreasing on $[\gamma''^{-1}(2^{l-1}),\gamma''^{-1}(2^{l-2})]$.
\end{enumerate}
Hence, with the notation used above,
\begin{equation}\label{sharp-mul}
\begin{split}
m(\xi_n,\eta_n) = \int_0^1 \frac{G'(t)}{\psi(t)} dt
&=
\int_0^{\gamma''^{-1}(2^{l+2})} \frac{G'(t)}{\psi(t)}(1 - h(t)) dt\\
&\hspace*{.3in} + 
\int_{\gamma''^{-1}(2^{l+3})}^{\gamma''^{-1}(2^{l-2})} \frac{G'(t)}{\psi(t)}
h(t) dt\\
&\hspace*{.5in} + 
\int_{\gamma''^{-1}(2^{l-1})}^1 \frac{G'(t)}{\psi(t)}(1 - h(t)) dt.
\end{split}
\end{equation}

To estimate the first integral, se integrate by parts to obtain
\begin{equation}\label{sharp-first-int}
\begin{split}
\int_0^{\gamma''^{-1}(2^{l+2})} &\frac{G'(t)}{\psi(t)}(1 - h(t)) dt\\
&=
- \lim_{t\to0} \frac{G(t)}{\psi(t)} +
\int_0^{\gamma''^{-1}(2^{l+2})} \frac{G(t)\psi'(t)}{\psi(t)^2}(1 - h(t))dt\\
&\hspace{.5in} -
\int_{\gamma''^{-1}(2^{l+3})}^{\gamma''^{-1}(2^{l+2})} \frac{G(t)}{\psi(t)} h'(t)dt.
\end{split}
\end{equation}

As in Section \ref{first-int}, if $0 < t < \gamma''^{-1}(2^{l+2})$,
\begin{equation*}
\gamma''(t) \ge 2^{l+2} = 2\gamma''(\gamma''^{-1}(2^{l+1}))
\ge 2\gamma''(t_n) = 4\eta_n,
\end{equation*}
so
\begin{equation*}
g''(t) = \gamma''(t) -2\eta_n \ge \frac{1}{2}\gamma''(t)
\end{equation*}
and, by van der Corput's lemma,
\begin{equation*}
|G(t)| \lesssim \frac{1}{\gamma''(t)^{1/2}}
\end{equation*}
and, by hypothesis (b.4),
\begin{equation*}
\Big| \frac{G(t)}{\psi(t)}\Big| \lesssim \frac{1}{\gamma''(t)^\rho}.
\end{equation*}
This implies that the limit in \eqref{sharp-first-int} is zero, as well as 
the third integral is bounded, as
\begin{multline*}
\Big|\int_{\gamma''^{-1}(2^{l+3})}^{\gamma''^{-1}(2^{l+2})} 
\frac{G(t)}{\psi(t)} h'(t)dt\Big| \\
\le \int_{\gamma''^{-1}(2^{l+3})}^{\gamma''^{-1}(2^{l+2})}
\Big|\frac{G(t)}{\psi(t)}\Big| h'(t)dt \lesssim \frac{1}{2^l}
\int_{\gamma''^{-1}(2^{l+3})}^{\gamma''^{-1}(2^{l+2})}h'(t)dt = 2^{-l},
\end{multline*}
because $h$ is increasing from $0$ to $1$ in the interval
$[\gamma''^{-1}(2^{l+3}),\gamma''^{-1}(2^{l+2})]$.

For the second integral in \eqref{sharp-first-int}, we decompose the interval
$(0,\gamma''^{-1}(2^{l+2})]$ and use the fact that $\psi$
is monotone, as in Section \ref{section-neg-eta}, to obtain
\begin{equation*}
\begin{split}
\Big| \int_0^{\gamma''^{-1}(2^{l+2})} 
&\frac{G(t)\psi'(t)}{\psi(t)^2}(1 - h(t))dt \Big|
\le 
\sum_{j=l+2}^\infty 
\int_{\gamma''^{-1}(2^{j+1})}^{\gamma''^{-1}(2^j)} |G(t)|
\Big| \frac{\psi'(t)}{\psi(t)^2} \Big| dt \\
&\lesssim \sum_{j=l+2}^\infty 
\int_{\gamma''^{-1}(2^{j+1})}^{\gamma''^{-1}(2^j)} \frac{1}{\gamma''(t)^{1/2}}
\Big| \frac{\psi'(t)}{\psi(t)^2} \Big| dt \\
&\le \sum_{j=l+2}^\infty \frac{1}{2^{j/2}}
\int_{\gamma''^{-1}(2^{j+1})}^{\gamma''^{-1}(2^j)} 
\Big| \frac{\psi'(t)}{\psi(t)^2} \Big| dt \\
&\le \sum_{j=l+2}^\infty \frac{1}{2^{j/2}}
\Big| \frac{1}{\psi(\gamma''^{-1}(2^j))} -
\frac{1}{\psi(\gamma''^{-1}(2^{j+1}))} \Big| \\
&\lesssim \sum_{j=l+2}^\infty \frac{1}{2^{j/2}}
\gamma''(\gamma''^{-1}(2^{j+1}))^{1/2-\rho}\\
&\lesssim \sum_{j=l+2}^\infty \frac{1}{2^{j/2}} (2^j)^{1/2-\rho}
= \sum_{j=l+2}^\infty \frac{1}{2^{j\rho}} < \infty.
\end{split}
\end{equation*}
The third integral in \eqref{sharp-mul} can be estimated similarly. We conclude
then that both of these integrals are bounded.

Now, since $g'(t_n) = g''(t_n) = 0$, we have the asymptotic estimate
\begin{equation*}
\int_{\gamma''^{-1}(2^{l+3})}^{\gamma''^{-1}(2^{l-2})}
\frac{e^{2\pi i g(t)}}{\psi(t)} h(t) dt \sim
C \frac{e^{2\pi ig(t_n)}}{\psi(t_n)(\gamma'''(t_n))^{1/3}}
\end{equation*}
for the second integral in \eqref{sharp-mul}, where $C$ is independent of $n$
\cite{Erdelyi}. By the construction of the $t_n$, we have that
\begin{equation*}
|m(\xi_n,\eta_n)| \to \infty
\end{equation*}
as $t\to\infty$.
\end{proof}

\section[An $L^p$ result]{An extension to $L^p$ spaces}

In this section we consider the operator 
\begin{equation}\label{operTheta}
Tf(x,y) = \lim_{\e\to 0} \int_\e^1 f(x - t, y - t^k)
\frac{e^{2\pi i\gamma(t)}}{t^\theta \, \psi(t)^{1 - \theta}} dt,
\end{equation}
for $0< \theta < 1$, and $\gamma$ and $\psi$ satisfy hypotheses 
(a.1)-(a.4) above, as well as assuming that $\psi > 0$. We also make the 
following assumption on $\gamma''$:
\begin{enumerate}
\item[(a.5)] There exists a constant $\epsilon>0$, with $(1+\epsilon)^2 < 2$,
such that, for $0 < t < t_0$, 
\begin{equation*}
\gamma''(t) \ge 2 \gamma''((1+\epsilon)t).
\end{equation*}
\end{enumerate}
We remark that assumption (a.5) is analogous to assumption (a.4) in 
\cite{FolchSaenz19}, which stated that 
\begin{itemize}
\item[] ``there exist $\epsilon > 0, A > 1 + \epsilon$ such that 
$\gamma'(t) \ge A \gamma'((1+\epsilon)t)$ for for $0 < t < t_0$''.
\end{itemize}
It is expected, as we are now working along the curve $t\mapsto(t,t^k)$, to
require a stronger condition to the corresponding one in \cite{FolchSaenz19}, so
this time we need to assume such condition on $\gamma''$.

By substituting $(1+\epsilon)t$ by $\gamma''(s)$ we get
\begin{equation}\label{doublinggamma2}
\gamma''^{-1}(s) \le (1+\epsilon) \gamma''^{-1}(2s)
\end{equation}
for large $s>1$. Together with the fact that $\gamma''^{-1}$ is decreasing we
obtain that 
\begin{equation}\label{gamma2comp}
\gamma''^{-1}(2s) \le \gamma''^{-1}(s) \lesssim \gamma''^{-1}(2s)
\end{equation}
for large values of $s$, so $\gamma''^{-1}(s)$ and $\gamma''^{-1}(2s)$ have
comparable sizes. This is analogous to the fact that 
$|\gamma'''(\gamma''^{-1}(s))|$ and $|\gamma'''(\gamma''^{-1}(2s))|$ also have
comparable sizes, as, from assumption (a.2), we have that 
\begin{equation}\label{gamma3comp}
|\gamma'''(\gamma''^{-1}(s))| \le |\gamma'''(\gamma''^{-1}(2s))|
\lesssim |\gamma'''(\gamma''^{-1}(s))|.
\end{equation}

\begin{remark}
The function $\gamma(t) = t^{-\beta}$, with $\beta>0$, also satisfies
assumption (a.5). Indeed, $\gamma''(t) = \beta(\beta+1)t^{-\beta-2}$, so
\begin{equation*}
\gamma''(t) = (1+\epsilon)^{\beta+2} \gamma''((1+\epsilon)t) 
\ge 2 \gamma''((1+\epsilon)t)
\end{equation*}
if we choose $(1+\epsilon)^{\beta+2} = 2$, which satisfies $(1+\epsilon)^2 < 2$
because $\beta > 0$.
\end{remark}

\begin{remark}
Functions $\gamma$ such that $\gamma''$ is of the form 
$\gamma''(t) = c e^{\sigma/t}$, for $\sigma>0$, also satisfy (a.5).
Indeed, 
\begin{equation*}
\gamma''(t) =
e^{\frac{\sigma}{t}\frac{\epsilon}{1+\epsilon}}\gamma''((1+\epsilon)t)
\ge 2 \gamma''((1+\epsilon)t)
\end{equation*}
for $\epsilon = \log 2^{1/3} < 2^{1/2}-1$ as long as 
$t < \sigma/\big( 3(1+\epsilon)\big)$.
\end{remark}

\begin{lemma}\label{gamma3gamma2comp}
There exist constants $C>0$ and $\delta>0$ such that, for small $t>0$,
\begin{equation*}
|\gamma'''(t)| \ge C \frac{\gamma''(t)^\delta}{t^3}.
\end{equation*}
\end{lemma}

\begin{proof}
We first procced as in \cite{Folch99} to verify that
\begin{equation}\label{gamma3vsgamma2}
|\gamma'''(t)| \gtrsim \frac{\gamma(t)}{t}.
\end{equation}
Indeed, as $t\mapsto -\gamma'''(t)$ is decreasing we have
\begin{equation*}
-\gamma'''(t)\cdot \epsilon t \ge \int_t^{(1+\epsilon)t} -\gamma'''(s)ds = 
\gamma''(t) - \gamma''((1+\epsilon)t) \ge \frac{1}{2}\gamma''(t)
\end{equation*}
by assumption (a.5). So we obtain
\begin{equation*}
|\gamma'''(t)| \ge \frac{1}{2\epsilon} \frac{\gamma''(t)}{t}.
\end{equation*}

The lemma will follow from \eqref{gamma3vsgamma2} once we establish the
existence of $\delta>0$ such that
\begin{equation}\label{gamma2delta}
\gamma''(t) \gtrsim \frac{\gamma''(t)^\delta}{t^2}.
\end{equation}
To verify this, observe that, by assumption (a.5),
\begin{equation*}
\gamma''\Big( \frac{t_0}{(1+\epsilon)^n}\Big) \ge 2^n \gamma''(t_0) = 
C \Big(\frac{2}{1+\epsilon} \Big)^{n+1} \frac{1}{t_0/(1+\epsilon)^{n+1}}
\end{equation*}
for every $n\in\N$, where $C = t_0\gamma''(t_0)/2$. For $t\in(0,t_0)$, choose
$n$ such that
\begin{equation*}
\frac{t_0}{(1+\epsilon)^{n+1}} \le t < \frac{t_0}{(1+\epsilon)^n}.
\end{equation*}
Then $(1+\epsilon)^{n+1} \ge t_0/t$ and hence 
$n+1 \ge \log(t_0/t)/\log(1+\epsilon)$, which implies
\begin{equation*}
\Big(\frac{2}{1+\epsilon} \Big)^{n+1} \ge \Big(\frac{t_0}{t}\Big)^\alpha,
\end{equation*}
where
\begin{equation*}
\alpha = \frac{\log(2/(1+\epsilon))}{\log(1+\epsilon)} > 1
\end{equation*}
because $2 < (1+\epsilon)^2$. Therefore
\begin{equation}
\gamma''(t) > \gamma''\Big( \frac{t_0}{(1+\epsilon)^n}\Big) \ge
C \cdot \Big(\frac{t_0}{t}\Big)^\alpha \cdot \frac{1}{t} 
= \frac{C'}{t^{1+\alpha}}.
\end{equation}
If we choose $\delta = (\alpha-1)/(1+\alpha)$ we obtain
\begin{equation*}
\gamma''(t) > \gamma''(t)^\delta \cdot 
\Big( \frac{C'}{t^{1+\alpha}} \Big)^{1-\delta} 
= C'' \frac{\gamma''(t)^\delta}{t^2}.
\end{equation*}
\end{proof}

\begin{remark}\label{seriescon}
Substituting $t$ by $\gamma''^{-1}(s)$ in Lemma \ref{gamma3gamma2comp} we
obtain
\begin{equation*}
\gamma''^{-1}(s) |\gamma'''(\gamma''^{-1}(s))|^{1/3} \gtrsim s^{\delta/3},
\end{equation*}
which guarantees, in particular, that for any $\e>0$ the series
\begin{equation*}
\sum_{j=0}^\infty \frac{1}{\big(
\gamma''^{-1}(2^j) |\gamma'''(\gamma''^{-1}(2^j))|^{1/3}\big)^\e} < \infty.
\end{equation*}
\end{remark}

The existence of the limit in \eqref{operTheta} can be verified analogously
to the existence of the limit in \eqref{operator} above. From the proof of
Theorem \ref{thm}, in particular, the estimates for the multiplier
\eqref{mult-int}, we obtain the following result.

\begin{corollary}\label{corp}
Under the assumptions (a.1-a.5), the operator \eqref{operTheta} is bounded on
$L^p(\R^2)$ if
\begin{equation}\label{condTheta}
\frac{1}{2}\le \frac{1}{p} < \frac{1+\theta}{2}.
\end{equation}
\end{corollary}

\begin{remark}
We take a moment to compare our result in Corollary \ref{corp} with
Chandarana's result \cite{Chandarana} for the operator \eqref{Tab}. If we
write $t^{\alpha+1}$ as $t^\theta \psi(t)^{1-\theta}$, with
\begin{equation*}
\psi(t) = t^{\frac{\alpha}{1-\theta} + 1},
\end{equation*}
then the condition $\beta \ge 3\alpha/(1-\theta)$ implies 
\begin{equation*}
\theta < \frac{\beta - 3\alpha}{\beta},
\end{equation*}
so we see that the operator \ref{Tab} is bounded on $L^p(\R^2)$ if
$\beta > 3\alpha$ and
\begin{equation*}
\frac{1}{p} < \frac{1}{2} + \frac{\beta - 3\alpha}{\beta}.
\end{equation*}
\end{remark}

\begin{proof}[Proof of Corollary \ref{corp}.]
We follow the idea in \cite{CDZ17} and formally write 
$Tf = \sum_{j=0}^\infty T_jf$, where 
\begin{equation*}
T_jf(x,y) = \int_{\gamma''^{-1}(2^{j+1})}^{\gamma''^{-1}(2^j)} f(x-t, y-t^k) 
\frac{e^{2\pi i\gamma(t)}}{t^\theta \, \psi(t)^{1 - \theta}} dt,
\end{equation*}
and we estimate the $L^1$ and $L^2$ norms of each $T_j$, $j\ge 0$. 

For the $L^1$ norms, we start with the trivial estimate
\begin{equation*}
\begin{split}
\int |T_jf(x,y)|dxdy &\le \int_{\gamma''^{-1}(2^{j+1})}^{\gamma''^{-1}(2^j)}
\frac{1}{t^\theta \, \psi(t)^{1 - \theta}} dt ||f||_{L^1}\\
&\lesssim \frac{\gamma''^{-1}(2^j)}{\gamma''^{-1}(2^{j+1})^\theta
\psi(\gamma''^{-1}(2^{j+1}))^{1 - \theta}} ||f||_{L^1}.
\end{split}
\end{equation*}
Using (a.4) and \eqref{gamma3comp} we have
\begin{equation*}
\frac{1}{\psi(\gamma''^{-1}(2^{j+1}))^{1-\theta}} \lesssim 
|\gamma'''(\gamma''^{-1}(2^{j+1}))|^{(1-\theta)/3} \lesssim
|\gamma'''(\gamma''^{-1}(2^j))|^{(1-\theta)/3}
\end{equation*}
and, by \eqref{gamma2comp}, we have 
\begin{equation*}
\gamma''^{-1}(2^{j+1}) \gtrsim \gamma''^{-1}(2^j),
\end{equation*}
so we have the estimate
\begin{equation}\label{L1est}
||T_j||_{L^1\to L^1} \lesssim 
\gamma''^{-1}(2^j)^{1-\theta}|\gamma'''(\gamma''^{-1}(2^j))|^{(1-\theta)/3}.
\end{equation}

For the $L^2$ norm, we estimate the supremum of its multiplier
\begin{equation*}
m_j(\xi,\eta) = \int_{\gamma''^{-1}(2^{j+1})}^{\gamma''^{-1}(2^j)}
\frac{G'(t)}{t^\theta \psi(t)^{1-\theta}} dt,
\end{equation*}
with the same notation as above. Integrating by parts, 
\begin{multline}\label{MultInParts}
m_j(\xi,\eta) = \frac{G(\gamma''^{-1}(2^j))}{\gamma''^{-1}(2^j)^\theta \psi(\gamma''^{-1}(2^j))^{1-\theta}} -
\frac{G(\gamma''^{-1}(2^{j+1})}{\gamma''^{-1}(2^{j+1})^\theta \psi(\gamma''^{-1}(2^{j+1})^{1-\theta}} \\
+ \int_{\gamma''^{-1}(2^{j+1})}^{\gamma''^{-1}(2^j)} G(t) \Big(
\frac{\theta}{t^{\theta+1} \psi(t)^{1-\theta}} + 
\frac{(1-\theta)\psi'(t)}{t^\theta \psi(t)^{2-\theta}} \Big) dt.
\end{multline}
Again, by Van der Corput's lemma, 
\[
|G(t)| \lesssim \frac{1}{|\gamma'''(t)|^{1/3}}.
\]
Using (a.4) again, each of the first two terms in \eqref{MultInParts} is
estimated by
\begin{equation*}
\begin{split}
\Big| \frac{G(\gamma''^{-1}(2^j))}{\gamma''^{-1}(2^j)^\theta
\psi(\gamma''^{-1}(2^j))^{1-\theta}} \Big| &\lesssim
\frac{|\gamma'''(\gamma''^{-1}(2^j))|^{(1-\theta)/3}}{\gamma''^{-1}(2^j)^\theta
|\gamma'''(\gamma''^{-1}(2^j))|^{1/3}} \\ &\lesssim
\frac{1}{\gamma''^{-1}(2^j)^\theta
|\gamma'''(\gamma''^{-1}(2^j))|^{\theta/3}}
\end{split}\end{equation*}
and, also using \eqref{gamma2comp} and \eqref{gamma3comp}, 
\begin{equation*}
\begin{split}
\Big| \frac{G(\gamma''^{-1}(2^{j+1}))}{\gamma''^{-1}(2^{j+1})^\theta
\psi(\gamma''^{-1}(2^{j+1}))^{1-\theta}} \Big| &\lesssim
\frac{|\gamma'''(\gamma''^{-1}(2^{j+1}))|^{(1-\theta)/3}}{
\gamma''^{-1}(2^{j+1})^\theta
|\gamma'''(\gamma''^{-1}(2^{j+1}))|^{1/3}} \\ &\lesssim
\frac{1}{
\gamma''^{-1}(2^j)^\theta|\gamma'''(\gamma''^{-1}(2^j))|^{\theta/3}}.
\end{split}
\end{equation*}

Similarly, we estimate the integral by
\begin{equation*}
\begin{split}
\Big|\int_{\gamma''^{-1}(2^{j+1})}^{\gamma''^{-1}(2^j)} 
\frac{G(t)\theta}{t^{\theta+1} \psi(t)^{1-\theta}} dt \Big| &
\le
\int_{\gamma''^{-1}(2^{j+1})}^{\gamma''^{-1}(2^j)}
\frac{|G(t)|}{t^{\theta+1} \psi(t)^{1-\theta}} dt \\
&\lesssim \int_{\gamma''^{-1}(2^{j+1})}^{\gamma''^{-1}(2^j)}
\frac{|\gamma'''(t)|^{(1-\theta)/3}}{t^{\theta+1} |\gamma'''(t)|^{1/3}} dt\\
&\lesssim \frac{1}{|\gamma'''(\gamma''^{-1}(2^{j+1}))|^{\theta/3}}
\int_{\gamma''^{-1}(2^{j+1})}^{\gamma''^{-1}(2^j)}
\frac{1}{t^{\theta+1}} dt \\
&\lesssim
\frac{1}{
\gamma''^{-1}(2^j)^\theta|\gamma'''(\gamma''^{-1}(2^j))|^{\theta/3}},
\end{split}
\end{equation*}
where have again used \eqref{gamma3comp}, and, using the fact that $\psi$ is
monotone (and thus increasing),
\begin{equation*}
\begin{split}
\Big|\int_{\gamma''^{-1}(2^{j+1})}^{\gamma''^{-1}(2^j)} 
&\frac{G(t)(1-\theta)\psi'(t)}{t^\theta \psi(t)^{2-\theta}} dt \Big|
\le
\int_{\gamma''^{-1}(2^{j+1})}^{\gamma''^{-1}(2^j)}
\frac{|G(t)|}{t^\theta}
\frac{\psi'(t)}{\psi(t)^{2-\theta}} dt \\
&\lesssim \frac{1}{
\gamma''^{-1}(2^{j+1})^\theta|\gamma'''(\gamma''^{-1}(2^{j+1}))|^{1/3}}
\int_{\gamma''^{-1}(2^{j+1})}^{\gamma''^{-1}(2^j)} 
\frac{\psi'(t)}{\psi(t)^{2-\theta}} dt \\
&\lesssim \frac{1}{
\gamma''^{-1}(2^{j+1})^\theta|\gamma'''(\gamma''^{-1}(2^{j+1}))|^{1/3}}
\frac{1}{\psi(\gamma''^{-1}(2^{j+1}))^{1-\theta}}\\
&\lesssim \frac{1}{
\gamma''^{-1}(2^{j+1})^\theta|\gamma'''(\gamma''^{-1}(2^{j+1}))|^{\theta/3}}\\
&\lesssim
\frac{1}{
\gamma''^{-1}(2^j)^\theta|\gamma'''(\gamma''^{-1}(2^j))|^{\theta/3}}.
\end{split}
\end{equation*}
We therefore conclude that
\begin{equation}\label{L2est}
||T_j||_{L^2\to L^2} \lesssim \frac{1}{
\gamma''^{-1}(2^j)^\theta|\gamma'''(\gamma''^{-1}(2^j))|^{\theta/3}}.
\end{equation}

By interpolation, if $0\le\tau\le1$ and
\begin{equation*}
\frac{1}{p} = \tau + \frac{1-\tau}{2},
\end{equation*}
then $T_j$ is bounded on $L^p(\R^2)$ with norm
\begin{equation}\label{Tjnorm}
\begin{split}
||T_j||_{L^p\to L^p} &\le ||T_j||_{L^1\to L^1}^\tau
||T_j||_{L^2\to L^2}^{1-\tau}\\ & \lesssim
\frac{1}{\gamma''^{-1}(2^j)^{\theta-\tau}
|\gamma'''(\gamma''^{-1}(2^j))|^{(\theta-\tau)/3}}.
\end{split}
\end{equation}
By Remark \ref{seriescon}, the series
\begin{equation*}
\sum_{j=0}^\infty ||T_j||_{L^p\to L^p} \lesssim
\sum_{j=0}^\infty \frac{1}{\gamma''^{-1}(2^j)^{\theta-\tau}
|\gamma'''(\gamma''^{-1}(2^j))|^{(\theta-\tau)/3}}
\end{equation*}
converges as long as $\theta > \tau$, that is
\begin{equation*}
\frac{1}{p} = \frac{1+\tau}{2} < \frac{1+\theta}{2}.
\end{equation*}
\end{proof}

\section{Discussion}

In \cite{Folch99}, the author proved that the operator
\begin{equation}\label{line}
f \mapsto \lim_{\e\to 0} \int_\e^1 f(x - t)
\frac{e^{2\pi i\gamma(t)}}{t^\theta \, \psi(t)^{1 - \theta}} dt
\end{equation}
is bounded on $L^p(\R)$ as long as $\gamma, \psi$ satisfy assumptions that are
analogous to our assumptions (a.1)-(a.5) (as discussed in this article) and $p$
satisfies 
\begin{equation*}
\frac{1}{2} \le \frac{1}{p} \le \frac{1+\theta}{2},
\end{equation*}
for $0\le \theta < 1$. Thus, a natural question is whether Corollary \ref{corp}
is true for the endpoint case $p=(1+\theta)/2$.


\end{document}